\documentclass[a4paper, final, 11pt]{article}

\usepackage{amsfonts,enumerate} 
\usepackage{enumerate}
\usepackage{graphicx,psfrag}
\newtheorem{proposition}{Proposition}[section]
\newtheorem{theorem}[proposition]{Theorem}
\newtheorem{lemma}[proposition]{Lemma}
\newtheorem{corollary}[proposition]{Corollary}
\newtheorem{definition}[proposition]{Definition}
\newtheorem{remark}[proposition]{Remark}

\newenvironment{proof}{\smallskip\noindent\emph{\textbf{Proof.}}\hspace{1pt}}%
{\hspace{-5pt}{\nobreak\quad\nobreak\hfill\nobreak$\square$\vspace{8pt}%
\par}\smallskip\goodbreak}

\newenvironment{proofof}[1]{\smallskip\noindent\emph{\textbf{Proof of #1.}}%
\hspace{1pt}}{\hspace{-5pt}{\nobreak\quad\nobreak\hfill\nobreak%
$\square$\vspace{8pt}\par}\smallskip\goodbreak}

\newcommand{\Section}[1]{\section{#1}\setcounter{equation}{0}}

\newcommand{\Id}{\mathinner{\mathrm{Id}}}

\newcommand{\comp}{\mathop\bigcirc}

\newcommand{\C}[1]{\mathbf{C^{#1}}}
\newcommand{\CB}[1]{\mathbf{C_B^{#1}}}

\newcommand{\modulo}[1]{{\left|#1\right|}}
\newcommand{\norma}[1]{{\left\|#1\right\|}}
\newcommand{\Ref}[1]{{\rm(\ref{#1})}}
\newcommand{\reali}{{\mathbb{R}}}
\newcommand{\interi}{{\mathbb{Z}}}
\newcommand{\naturali}{{\mathbb{N}}}

\newcommand{\Lip}{\mathinner\mathbf{Lip}}

\renewcommand{\epsilon}{\varepsilon}
\renewcommand{\phi}{\varphi}
\renewcommand{\theta}{\vartheta}
\renewcommand{\L}[1]{\mathbf{L^#1}}

\title{Differential Equations in Metric Spaces \\ with Applications}

\author{Rinaldo M.~Colombo \\ \small Dipartimento di Matematica \\
  \small Universit\`a degli Studi di Brescia \\ \small Via Branze, 38
  \\ \small 25123 Brescia, Italy \\ \texttt{Rinaldo.Colombo@UniBs.it} \\
  \and Graziano Guerra \\ \small Dip.~di Matematica e Applicazioni \\
  \small Universit\`a di Milano -- Bicocca \\ \small Via Bicocca degli
  Arcimboldi, 8 \\ \small 20126 Milano, Italy \\
  \texttt{Graziano.Guerra@UniMiB.it}}

\begin{document}

\maketitle

\begin{abstract}

  \noindent This paper proves the local well posedness of differential
  equations in metric spaces under assumptions that allow to comprise
  several different applications. We consider below a system of
  balance laws with a dissipative non local source, the Hille-Yosida
  Theorem, a generalization of a recent result on nonlinear operator
  splitting, an extension of Trotter formula for linear semigroups and
  the heat equation.

  \medskip

  \noindent\textit{2000~Mathematics Subject Classification:} 34G20,
  47J35, 35L65.

  \medskip


\end{abstract}

\Section{Introduction}
\label{sec:Intro}

This paper is concerned with differential equations in a metric space
and their applications to ordinary and partial differential equations.
More precisely, we modify the structure of quasidifferential equations
introduced in~\cite{Panasyuk, Panasyuk97}, see
also~\cite{BressanLarnas}, and prove existence, uniqueness and
continuous dependence under conditions weaker than those therein. A
general estimate on the difference between any Euler polygonal
approximation and the exact solution is also provided.

Following the model theory of ordinary differential equations, we
provide a result that comprises also other entirely different
examples: the heat equation, the Hille-Yosida theorem, the
Lie--Trotter product formula and a balance law with a dissipative non
local source. In particular, the latter result is an extension
of~\cite{ColomboGuerra} that was announced in~\cite{GuerraColombo}.
Furthermore, the construction below weakens the assumptions introduced
in~\cite{ColomboCorli5} on the non linear operator splitting technique
in metric spaces, as well as those introduced in~\cite{KuhneWacker1}
on linear Lie-Trotter products. Remark that the commutativity
condition adopted here, namely~\Ref{eq:k}, is optimal, in the sense of
Paragraph~\ref{sub:example}.

Following~\cite{BressanLarnas, Panasyuk, Panasyuk97}, for any $u$ in
the metric space $X$, the tangent space $T_u X$ to $X$ at $u$ is the
quotient of the set $\left\{ \gamma \in \C{0,1} \left( [0,1]; X
  \right) \colon \gamma(0) = u \right\}$ of continuous curves exiting
$u$ modulo the equivalence relation of first order contact,
i.e.~$\gamma_1 \sim \gamma_2$ if and only if $\lim_{t \to 0}
\frac{1}{\tau} d \left( \gamma_1(\tau),\gamma_2(\tau) \right) = 0$.
Any map $v \colon [0,T] \times X \mapsto \bigcup_{u\in X} T_u X$, such
that $v(t,u) \in T_u X$ for all $t$ and $u$ then defines both a vector
field on $X$ and the (quasi)differential equation
\begin{equation}
  \label{eq:QDE}
  \dot u = v(t,u) \,.
\end{equation}
Besides suitable regularity assumptions, a solution to~\Ref{eq:QDE} is
a process $(t,t_o,u) \mapsto P(t,t_o)u$ such that $P(0,t_o) u_o = u_o$
and the curve $\tau \to P(t+\tau,t_o) u_o$ belongs to the equivalence
class of $v\left( t_o+t, P(t,t_o)u \right)$, for all $t \geq 0$. A
\emph{local flow} generated by $v$ is a map $(t,t_o,u) \mapsto
F(t,t_o)u$ such that $F(0,t_o) u_o = u_o$ and the curve $\tau \to
F(\tau,t_o) u_o$ belongs to the equivalence class of $v\left( t_o, u_o
\right)$.

Within this framework, we prove that if the vector field $v$ can be
defined through a suitable \emph{local flow}, then the Cauchy
problem~\Ref{eq:QDE} is well posed globally in $u_o \in X$ and locally
in $t \in \reali$. The proof is constructive and based on Euler
polygonal, similarly to~\cite{Panasyuk, Panasyuk97}, but with somewhat
looser assumptions.

This approach to differential equations, being sited in metric spaces,
does not rely on any linearity assumption whatsoever. It is therefore
particularly suited to describe truly nonlinear or even discontinuous
models. In this connection, we refer to the example in
Paragraph~\ref{subsec:stop} below and to~\cite[Theorem~3.1]{Panasyuk},
which is contained in the present framework.

The next section is the core of this paper: it presents the definition
of \emph{local flow} and Theorem~\ref{thm:main}, which shows that a
local flow generates a global process. We also give an example showing
the necessity of the assumptions in Theorem~\ref{thm:main}.
Section~\ref{sec:Applications} presents several applications. The
first one deals with a recent new result about balance laws with a
diagonally dominant non local source.  The final
Section~\ref{sec:Tech} provides the various technical details.

\Section{Notation and Main Result}
\label{sec:Result}

Throughout, $(X,d)$ denotes a complete metric space. In view of the
applications in Section~\ref{sec:Applications}, we need to slightly
extend the basic definitions about differential equations in metric
spaces in~\cite{BressanLarnas}, see also~\cite{Aubin, Panasyuk,
  Panasyuk97}.

Throughout this paper, we fix $T \in \left[0, +\infty\right[$ and the
interval $I = [0,T]$.

\begin{definition}
  \label{def:local}
  Given a closed set $\mathcal{D} \subseteq X$, a \emph{local flow} is
  a continuous map $F \colon [0,\delta] \times I \times \mathcal{D}
  \mapsto X$, such that $F\left(0,t_o\right)u=u$ for any $(t_o,u)\in
  I\times \mathcal{D}$ and which is Lipschitz in its first and third
  argument uniformly in the second, i.e.~there exists a 
  $\Lip(F) > 0$ such that for all $\tau,\tau' \in [0,\delta]$ and
  $u,u' \in \mathcal{D}$
  \begin{equation}
    d \left( F(\tau,t_o)u, F(\tau',t_o) u' \right)
    \leq
    \Lip(F) \cdot \left( d(u,u') + \modulo{\tau - \tau'}\right) \,.
  \end{equation}
\end{definition}

\noindent To explain the notation, consider the case of a Banach space
$X$ and let $\mathcal{D} \subseteq X$. If $v \colon I \times
\mathcal{D} \mapsto X$ is a vector field defining the ordinary
differential equation $\dot u = v(t,u)$, then a local flow (generated
by $v$) is the map $F$ where $F(\tau,t_o) u = u + \tau \, v(t_o,u)$.

\begin{definition}
  Let $F$ be a local flow. Fix $u \in \mathcal{D}$, $t_o \in I$, $t >
  0$ with $t_o+t \in I$, for any $k \in \naturali $ and $\tau_0 = 0$,
  $\tau_1, \ldots, \tau_k$, $\tau_{k+1} = t$ such that
  $\tau_{h+1}-\tau_h \in \left]0,\delta \right]$ for $h=0, \ldots, k$,
  we call Euler polygonal the map
  \begin{equation}
    \label{eq:Euler}
    F^E (t,t_o) \, u 
    = 
    \comp_{h=0}^{k} F(\tau_{h+1} -\tau_h, t_o + \tau_h) \, u
  \end{equation}
  whenever it is defined.  For any $\epsilon > 0$, let $k = [t /
  \epsilon]$. An Euler $\epsilon$-polygonal is
  \begin{equation}
    \label{eq:polygonal}
    F^\epsilon (t,t_o) \, u
    =
    F(t-k\epsilon, t_o+k\epsilon) \circ 
    \comp_{h=0}^{k-1} F(\epsilon, t_o+h\epsilon) \, u
  \end{equation}
  whenever it is defined.
\end{definition}

\noindent Above and in what follows, $[\,\cdot\,]$ denotes the integer
part, i.e.~for $s \in \reali$, $[s] = \max \{k \in \interi \colon k
\leq s \}$. Clearly, $F^E$ in~\Ref{eq:Euler} reduces to $F^\epsilon$
in~\Ref{eq:polygonal} as soon as $\tau_h = t_o + h \epsilon$ for $h=0,
\ldots, k=[t/\epsilon]$.  The Lipschitz dependence on time of the
local flow implies the same regularity of the Euler
$\epsilon$-polygonals, as shown by the next lemma, proved in
Section~\ref{sec:Tech}.

\begin{lemma}
  \label{lem:Time}
  Let $F$ be a local flow. Then $F^\epsilon$, whenever defined, is
  Lipschitz in $t,u$ and continuous in $t_o$. Moreover,
  \begin{displaymath}
    d \left( F^\epsilon(t,t_o)u, F^\epsilon(s,t_o)u\right) 
    \leq 
    \Lip(F) \, \modulo{t-s} \,.
  \end{displaymath}
\end{lemma}

On the other hand, an elementary computation shows that
\begin{displaymath}
  d \left( F^\epsilon (t,t_o)u, F^\epsilon (t,t_o)w \right)
  \leq
  \left(\Lip(F)\right)^{[t/\epsilon]+1} \, d(u,w)
\end{displaymath}
so that the limit $\epsilon \to 0$ is possible only when $\Lip(F) \leq
1$.  In the case of conservation laws, for instance, the usual
procedure to prove the well posedness is based on this key point: the
Lipschitz dependence of $F^\epsilon$ on $u$ is achieved through the
introduction of a functional
(see~\cite[Chapter~8]{BressanLectureNotes} and~\cite{ColomboGuerra2})
or a metric (see~\cite{BianchiniTemple, BressanColombo}) equivalent to
the $\L1$ distance $d$ and according to which the Euler polygonals
turn out to be non expansive. Therefore, in the theorem below, the
Lipschitzeanity of $F^\epsilon$ on $u$ is explicitly required.

With the same notation of Definition~\ref{def:local}, we introduce
what provides a solution to the Cauchy problem~\Ref{eq:QDE}.

\begin{definition}
  \label{def:Local}
  Consider a family of sets $\mathcal{D}_{t_o}\subseteq\mathcal{D}$
  for all $t_o\in I$, and a set
  \begin{displaymath}
    \mathcal{A} 
    =
    \left\{ 
      (t,t_o,u) \colon t \geq 0,\  t_o,t_o+t \in I 
      \mbox{ and }u \in \mathcal{D}_{t_o}
    \right\}
  \end{displaymath}
  A \emph{global process} on $X$ is a map $P \colon \mathcal{A}
  \mapsto X$ such that, for all $t_o,t_1,t_2,u$ satisfying $t_1,t_2\ge
  0$, $t_o,t_o+t_1+t_2\in I$ and $u\in\mathcal{D}_{t_o}$, satisfies
  \begin{eqnarray}
    \nonumber
    & & 
    P(0,t_o) u = u
    \\
    \label{eq:process0}
    & & 
    P(t_1,t_o)u\in\mathcal{D}_{t_o+t_1}
    \\
    \label{eq:process}
    & &
    P ( t_2, t_o + t_1) \circ P(t_1,t_o)u
    =
    P(t_2+t_1,t_o)u
  \end{eqnarray}
\end{definition}

To state the main theorem, we consider the following sets for any $t_o
\in I$
\begin{equation}
  \label{def:D3}
  \!\!\!
  \mathcal{D}^3_{t_o}
  =
  \left\{
    u\in\mathcal{D} \colon \!
    \begin{array}{l}
      F^{\epsilon_3} (t_3,t_o+t_1+t_2) \circ 
      F^{\epsilon_2} (t_2,t_o+t_1) \circ 
      F^{\epsilon_1} (t_1,t_o)u
      \\
      \mbox{is in } \mathcal{D}
      \mbox{ for all }
      \epsilon_1, \epsilon_2, \epsilon_3 \in
      \left]0,\delta \right]
      \mbox{ and all} 
      \\ t_1, t_2, t_3\geq 0 \mbox{ such that } t_o+t_1+t_2+t_3\in I
    \end{array}\!
  \right\}
\end{equation}

\begin{theorem}
  \label{thm:main}
  The local flow $F$ is such that there exists
  \begin{enumerate}
  \item \label{it:first} a non decreasing map $\omega \colon
    [0,\delta] \mapsto \reali^+$ with $\int_0^{\delta}
    \frac{\omega(\tau)}{\tau} \, d\tau < +\infty$ such that
    \begin{equation}
      \label{eq:k}
      d \left(
        F ( k \tau, t_o+\tau) \circ F(\tau,t_o) u, 
        F \left( (k+1) \tau, t_o \right)u
      \right)
      \leq
      k\tau \, \omega(\tau)
    \end{equation}
    whenever $\tau \in [0,\delta]$, $k\in \naturali$ and the left hand
    side above is well defined;
  \item \label{it:second} a positive constant $L$ such that
    \begin{equation}
      \label{eq:Stability}
      d \left( F^\epsilon (t,t_o) u, F^\epsilon (t,t_o) w \right)
      \leq L \cdot d(u,w)
    \end{equation}
    whenever $\epsilon \in \left]0, \delta\right]$, $u, w \in
    \mathcal{D}$, $t\geq 0$, $t_o,t_o+t\in I$ and the left hand side
    above is well defined.
  \end{enumerate}
  \noindent Then, there exists a family of sets $\mathcal{D}_{t_o}$
  and a unique a global process (as defined in
  Definition~\ref{def:local}) $P \colon \mathcal{A} \mapsto X$ with
  the following properties:
  \begin{enumerate}[a)]
  \item $\mathcal{D}^3_{t_o}\subseteq\mathcal{D}_{t_o}$ for any
    $t_o\in I$;
  \item $P$ is Lipschitz in $(t,t_o,u)$;
  \item
    \label{fgsfssald3}
    $P$ is tangent to $F$ in the sense that for all $u \in
    \mathcal{D}_{t_o}$, for all $t$ such that $t \in \left]0,
      \delta\right]$ and $t_o +t \in I$:
    \begin{equation}
      \label{eq:tangent}
      \frac{1}{t} \,
      d\left( P(t,t_o)u, F(t,t_o)u \right)
      \leq
      \frac{2L}{\ln 2}
      \int_0^t \frac{\omega(\xi)}{\xi} \, d\xi \,.
    \end{equation}
  \end{enumerate}
\end{theorem}

\noindent The proof is deferred to Section~\ref{sec:Tech}.

Observe that in the general formulation of Theorem~\ref{thm:main}, the
set $\mathcal{A}$ where $P$ is defined could be empty.  However, in
the applications, the following stronger condition, equivalent
to~\cite[Condition~2.]{NietoLopez}, is often satisfied:
\begin{description}
\item[(D)] There is a family $\hat\mathcal{D}_{t_o}$ of subsets of
  $\mathcal{D}$ satisfying $F(t,t_o) \hat\mathcal{D}_{t_o} \subseteq
  \hat\mathcal{D}_{t_o+t}$ for any $t \in [0,\delta]$ and $t_o,t_o+t
  \in I$
\end{description}
\noindent If~\textbf{(D)} holds, then obviously one has
$\hat\mathcal{D}_{t_o} \subseteq \mathcal{D}^3_{t_o} \subseteq
\mathcal{D}_{t_o}$ for any $t_o\in I$, providing a lower bound for the
set where the process is defined. If~\textbf{(D)} does not hold, then
a lower bound on $\mathcal{A}$ needs to be found exploiting specific
information on the considered situation, see
Paragraph~\ref{sub:Trotter}.

Observe that, by~\Ref{eq:tangent}, the curve $\tau\mapsto
P(\tau,t_o)u$ represents the same tangent vector as $\tau\mapsto
F(\tau,t_o)u$, for all $u$.

We remark that assumption~\ref{it:first} is satisfied, for instance,
when
\begin{enumerate}
\item $F$ is a process, i.e.~$F(s,t_o+t) \circ F(t,t_o) = F(s+t,t_o)$;
\item $F$ is defined combining two commuting Lipschitz semigroups
  through the operator splitting algorithm.
\end{enumerate}
\noindent Above, ``commuting'' is meant in the sense
of~\cite[\textbf{(C)}]{ColomboCorli5}, see
Paragraph~\ref{subsec:operator}

The present construction is similar to that in~\cite{BressanLarnas,
  CalcaterraBleecker, KuhneWacker1, Panasyuk, Panasyuk97}. On one
hand, here we need the function $\omega$ to estimate the speed of
convergence to $0$ in~\Ref{eq:k}, while in~\cite[(3.17)]{Panasyuk} or,
equivalently,~\cite[Condition~4.]{NietoLopez}, any
convergence to $0$ is sufficient. On the other hand, our
Lipschitzeanity requirement~\Ref{eq:Stability} is strictly weaker
than~\cite[(3.16)]{Panasyuk}, allowing to consider balance laws.
Indeed, the semigroup generated by conservation laws in general does
not satisfy~\cite[(3.16)]{Panasyuk}.

Once the global process is built and its properties proved, the
following well posedness result for Cauchy problems in metric spaces
is at hand.

\begin{corollary}
  Let the vector field $v$ in~\Ref{eq:QDE} be defined by a local flow
  satisfying~\ref{it:first}.~and~\ref{it:second}.~in
  Theorem~\ref{thm:main}.  Then, there exists a global process $P$
  whose trajectories are solutions to~\Ref{eq:QDE}. Moreover, for all
  $t_o \in \reali$ and $u_o \in X$ such that $u_o \in
  \mathcal{D}_{t_o}$, let $w \colon [t_o, t_o+\bar t] \mapsto X$ be a
  solution to~\Ref{eq:QDE} satisfying the initial condition
  $w(t_o)=u_o$. If
  \begin{enumerate}
  \item $w$ is Lipschitz,
  \item $w(s) \in \mathcal{D}_{s}$ for all $s \in [t_o,t_o+\bar t]$,
  \end{enumerate}
  then $w$ coincides with the trajectory of $P$ exiting $u_o$ at time
  $t_o$:
  \begin{displaymath}
    w(t_o+t) = P(t,t_o) u
    \qquad \mbox{ for all }t \in [0,\bar t] \,.
  \end{displaymath}
\end{corollary}

\noindent The proof is essentially as that of~\cite[Corollary~1,
\S~6]{BressanLarnas} (see also the proof
of~\cite[Theorem~2.9]{BressanLectureNotes}) and is omitted.

Here, moreover, we show that all Euler polygonals of $F$ converge to
$P$, providing the following estimate on the speed of convergence.

\begin{proposition}
  \label{prop:Euler}
  With the same assumptions of Theorem~\ref{thm:main}, fix
  $u\in\mathcal{D}_{t_o}$. Let $F^E$ be defined as
  in~\Ref{eq:Euler}. If $F^E(s,t_o)\in\mathcal{D}_{t_o+s}$ for any
  $s\in[0,t]$ and $\Delta = \max_h (\tau_{h+1} - \tau_h) \in \left]0,
    \delta \right]$, then the following error estimate holds:
  \begin{displaymath}
    d \left( F^E (t,t_o) u, P(t,t_o) u \right)
    \leq
    \frac{2L^2}{\ln 2} \, t \,
    \int_0^\Delta \frac{\omega(\xi)}{\xi} \, d\xi.
  \end{displaymath}
\end{proposition}

\subsection{A Condition Weaker than~\Ref{eq:k} Is Not Sufficient}
\label{sub:example}

We now show that requiring~\Ref{eq:k} for $k=1$ only:
\begin{equation}
  \label{eq:2}
  d \left( 
    F(\tau,t_o+\tau)\circ F(\tau,t_o)u , F(2\tau,t_o)u
  \right)
  \leq
  \tau\, \omega(\tau)
\end{equation}
does not allow to prove the semigroup property in
Theorem~\ref{thm:main}.

\begin{proposition}
  \label{prop:no}
  Let the hypotheses of Theorem~\ref{thm:main} hold with~\Ref{eq:k}
  replaced by~\Ref{eq:2}. Assume that~\textbf{(D)} holds. Then, there
  exists a map $P$ as in Theorem~\ref{thm:main}, but with the
  ``semigroup'' condition~\Ref{eq:process} replaced by the weaker:
  \begin{equation}
    \label{it:no:2}
    P(2t,t_o)u = P(t,t_o+t) \circ P(t,t_o)u, 
    \mbox{ for } u \in \mathcal{D}_{t_o},\ t\geq 0,\ t_o,t_o+2t\in I
  \end{equation}
  and only continuous (not necessarily Lipschitz) with respect to
  $(t,t_o)$.
\end{proposition}

\noindent The proof is deferred to Section~\ref{sec:Tech}. Note that
$P$ is \emph{not} necessarily a process. Indeed, we show by an example
that this property may fail, due to the fact that condition~\Ref{eq:k}
may fail for $k \geq 2$.

In the metric space $X=[0,1]$ with the usual Euclidean distance we
construct an example showing that the assumptions in
Proposition~\ref{prop:no}, namely~\Ref{eq:2} instead of~\Ref{eq:k}, do
not suffice to show that $P$ is a process.

Let $\phi$ be any non constant function satisfying
\begin{displaymath}
  \phi \in \C1\left([1,2];\reali\right) \mbox{ and }
  \left\{
    \begin{array}{l}
      \phi(1)=\phi(2)=0
      \\
      \phi'(1)=\phi'(2)=0
      \\
      \phi \leq 0
    \end{array}
  \right.
\end{displaymath}
and denote $k(t) = [\log_2 t]$. Define now
\begin{displaymath}
  f(t) = \left\{
    \begin{array}{l@{\ \mbox{ if }}rcl}
      1 & t & = & 0
      \\
      \exp\left( 2^{k(t)} \, \phi\left( 2^{-k(t)} t\right) \right)
      & t & > & 0
    \end{array}
  \right.
  \quad \mbox{ and }\quad
  F(t,t_o)u = f(t) u \,.
\end{displaymath}
The following lemma, listing the main properties of $f$ and $F$ above,
is immediate and its proof omitted.

\begin{lemma}
  \label{lem:phi}
  With the notation above, $f \in \C{0,1} \left( \left[ 0, +\infty
    \right[; \reali \right)$ and $f(2^nt) = f(t)^{2^n}$ for all $t\ge
  0$ and $n\in \naturali$. Choosing $\mathcal{D}_t =X$ for all $t>0$,
  the map $F$ satisfies Definition~\ref{def:local}, in particular
  \begin{displaymath}
    d\left( F(t,t_o)u, F(t',t_o')u' \right)
    \leq
    \left( \max_{[1,2]} \modulo{\phi'} \right) \modulo{t-t'} 
    + 
    \modulo{u-u'}
  \end{displaymath}
  so that $F$ is non expansive in $u$ and the Euler polygonals
  $F^\epsilon$ trivially satisfy~\Ref{eq:Stability}. Finally, $F$
  satisfies~\Ref{eq:2}.
\end{lemma}

We now prove that~\Ref{eq:k} is strictly stronger than~\Ref{eq:2} and
that the latter assumption may not guarantee the existence of a gobal
process. Since the local flow $F$ does not depend on $t_o$, we write
$F(t)$ for $F(t,t_o)$.

\begin{proposition}
  \label{prop:example}
  For any fixed $u \in X$, there does not exist any Lipschitz
  semigroup $P$ tangent to $F$ at $u$ in the sense
  of~\Ref{eq:tangent}.
\end{proposition}

\begin{proofof}{Proposition~\ref{prop:example}}
  By contradiction, let $P$ be a Lipschitz semigroup tangent to $F$ in
  the sense of~\Ref{eq:tangent}.

  Since with $P(t)u$ we denote the action of $P$ on the real number
  $u\in[0,1]$, to avoid confusion, we explicitly denote with the dot
  any product between real numbers.
  
  First, we verify that $P$ must be multiplicative, i.e.~, $P(t)u =
  P(t)1\cdot u$ (remember that $P(t)1$ is the action of $P$ on the
  number $1\in X=[0,1]$).  By~\cite[Theorem~2.9]{BressanLectureNotes}
  applied to the Lipschitz curve $\tau\mapsto w(\tau)=P(\tau)1\cdot u$
  \begin{eqnarray*}
    & &
    \modulo{P(t)u - P(t)1\cdot u} \leq
    \\
    & \leq &
    \Lip(P)\cdot \int_0^t 
    \liminf_{h\to 0+} 
    \frac{
      \modulo{P \left(h\right)\left( P(\tau)1 \cdot u\right)  
        - P\left(\tau+h\right)1\cdot u}}{h} \, d\tau
    \\
    & \leq &
    \Lip(P)\cdot \int_0^t 
    \limsup_{h\to 0+} 
    \frac{
      \modulo{P \left(h\right)\left( P(\tau)1 \cdot u\right)  
        - F \left(h\right)\left( P(\tau)1 \cdot u\right)}}{h} \, d\tau\\
    &   & +
    \Lip(P)\cdot \int_0^t 
    \limsup_{h\to 0+} 
    \frac{
      \modulo{F \left(h\right)\left( P(\tau)1 \cdot u\right)
        - P\left(\tau+h\right)1\cdot u}}{h} \, d\tau\\
    &  \leq&
    \Lip(P)\cdot \int_0^t 
    \limsup_{h\to 0+} 
    \frac{
      \modulo{\left(F \left(h\right)\circ P(\tau)1\right) \cdot u
        -\left( P(h)\circ P\left(\tau\right)1\right) \cdot u}}{h} \, d\tau\\
    &\leq& 0
  \end{eqnarray*}
  where we have used the tangency condition~\Ref{eq:tangent} and the
  equality
  \begin{displaymath}
    F \left(h\right)\left( P(\tau)1 \cdot u\right)=f (h)\cdot  P(\tau)1 \cdot u=\left(f \left(h\right)\cdot  P(\tau)1\right) \cdot u
    =\left(F \left(h\right)\circ P(\tau)1\right) \cdot u.
  \end{displaymath}

  We may now denote $p(t)=P(t)1$ so that $P(t)u=p(t)\cdot u$. The
  invariance of $[0,1]$ implies $p(t)=P(t)1 \in [0,1]$ for all $t$
  and, by~\Ref{eq:tangent},
  \begin{displaymath}
    \lim_{t\to 0} \frac{f(t)-p(t)}{t} = \lim_{t\to 0} \frac{F(t)1-P(t)1}{t}=0 \,.
  \end{displaymath}
  Moreover, for any $n\in \naturali$, applying Lagrange Theorem to the
  map $x \mapsto x^{2^n}$ on the interval between $f(2^{-n}t)^{2^n}$
  and $p(2^{-n}t)^{2^n}$,
  \begin{eqnarray*}
    \modulo{f(t) - p(t)}
    & = &
    \modulo{f(2^{-n}t)^{2^n} - p(2^{-n}t)^{2^n}}
    \\
    & \leq & 
    2^n \, \modulo{\bar x}^{2^n-1} \, \modulo{f(2^{-n}t) - p(2^{-n}t)}
    \\
    & \leq &
    \frac{\modulo{f(2^{-n}t) - p(2^{-n}t)}}{2^{-n}}
    \\
    & \to & 
    0 \quad \mbox{ for } n\to +\infty
  \end{eqnarray*}
  Hence $f=p$, but in general $f(t+s) \neq f(t)\,f(s)$, giving a
  contradiction.
\end{proofof}

\Section{Applications}
\label{sec:Applications}

In the autonomous case, we often omit the initial time $t_o$ writing
$F(\tau) u$ for $F(\tau,t_o) u$. Below, \textbf{(D)} is satisfied in
all but one case.

\subsection{Balance Laws}

Consider the following nonlinear system of balance laws:
\begin{equation}
  \label{eq:BL}
  \partial_t u + \partial_x f(u) = G(u)
\end{equation}
where $f \colon \Omega \mapsto \reali^n$ is the smooth flow of a
nonlinear hyperbolic system of conservation laws, $\Omega$ is a non
empty open subset of $\reali^n$ and $G \colon \L1 (\reali;\Omega)
\mapsto \L1 (\reali;\reali^n)$ is a (possibly) \emph{non local}
operator. \cite[Theorem~2.1]{ColomboGuerra} proves that, for small
times, equation~\Ref{eq:BL} generates a global process uniformly
Lipschitz in $t$ and $u$. This result fits in
Theorem~\ref{thm:main}. Indeed, setting $X = \L1(\reali; \Omega)$, it
is possible to show that $F(t)u = S_t u + t G\left(S_tu\right)$ is a
\emph{local flow} of the vector field defined by~\Ref{eq:BL}. Here,
$S$ is the \emph{Standard Riemann
  Semigroup}~\cite[Definition~9.1]{BressanLectureNotes} generated by
the left-hand side in~\Ref{eq:BL}.

Using essentially~\cite[Remark~4.1]{AmadoriGuerra2002} one can first
show that the map $F$ satisfies~\ref{it:first}.~in
Theorem~\ref{thm:main}.  Then, the functional defined
in~\cite[formula~(4.3)]{ColomboGuerra2} allows to prove
also~\ref{it:second}. The details are found in~\cite{ColomboGuerra3},
where the present technique is applied to the diagonally dominant
case~\cite[formula~(1.14)]{DafermosHsiao}, yielding a process defined
globally in time.

\subsection{Constrained O.D.E.s -- The Stop Problem}
\label{subsec:stop}

Let $X$ be a Hilbert space and fix a closed convex subset $C \subseteq
X$. For a positive $T$, let $f \colon [0,T] \times C \mapsto X$.
Consider the following constrained ordinary differential equation:
\begin{displaymath}
  \left\{
    \begin{array}{l}
      \dot u = f(t,u) \,,\       u(t) \in C
      \\
      u(t_o) = u_o
    \end{array}
  \right.
\end{displaymath}
with $u_o \in C$.  The \emph{stop problem} above was considered, for
instance, in~\cite{KrasnoPokro}, see
also~\cite[Section~4.1]{KloedenSadovskyVasilyeva}.

We denote by $\Pi\colon X \mapsto C$ the projection of minimal
distance, i.e.~$\Pi(x) = y$ if and only if $y\in C$ and $\norma{x-y} =
d(x,C)$. Recall that $\Pi$ is non expansive.

This problem fits in the present setting, for instance when $f$ and $C$
satisfy the following conditions:
\begin{enumerate}
\item[\textbf{(f)}] $f$ is bounded. Moreover, there exist $\omega
  \colon [0,T] \mapsto \left[0, \infty\right[$ with $\int_0^T\!
  \frac{\omega(\tau)}{\tau} d\tau < +\infty$ and an $L>0$ such that
  for all $t_1,t_2 \in [0,T]$ and $u_1,u_2 \in C$
  \begin{displaymath}
    \norma{f(t_1,u_1) - f(t_2,u_2)}
    \leq 
    \omega \left( \modulo{t_1-t_2} \right) 
    + 
    L \cdot \norma{u_1-u_2} \,.
  \end{displaymath}
\item[\textbf{(C)}] for a $d>0$, $\Pi \in \C2 \left( \overline{B(C,d)}
    \setminus C; \partial C \right)$ and $\displaystyle
  \sup_{\overline{B(C,d)} \setminus C} \norma{D^2\Pi} < +\infty$.
\end{enumerate}
\noindent As usual, we denote $B(C,d) = \left\{ x \in X \colon d(x,C)
  < d\right\}$.  By~\cite[Theorem~2]{Holmes}, if $X=\reali^n$ and $C$
is a compact convex set with $\C3$ boundary, then~\textbf{(C)} holds.

\begin{proposition}
  \label{prop:stop} Let~\textbf{(f)} and~\textbf{(C)} hold. Then, the
  map $F \colon [0,\delta] \times [0,T] \times C \mapsto C$, with
  $\delta=\min \left\{d /(2\sup\norma{f}), T\right\}$, defined by
  \begin{displaymath}
    F(t,t_o)u = \Pi \left(u + t \, f(t_o,u) \right)
  \end{displaymath}
  is a local flow that satisfies~\textbf{(D)} and the hypotheses of
  Theorem~\ref{thm:main}.
\end{proposition}

\begin{proof}
  The continuity of $F$, as well as its Lipschitzeanity in $t$ and
  $u$, is immediate. Condition~\textbf{(D)} holds with
  $\mathcal{D}_{t_o} =C$ for all $t_o$.

  Consider~\Ref{eq:k}. Fix $\tau,\tau' \in [0,\delta]$ and compute:
  \begin{eqnarray*}
    \Delta
    & = &
    \norma{
      F(\tau',t_o+\tau) \circ F(\tau,t_o) u - 
      F( \tau'+\tau,t_o) u
    }
    \\
    & \leq &
    \Big\|
    \Pi \left( \Pi\left(u+\tau f(t_o,u) \right) 
      + 
      \tau' f \left(t_o+\tau, \Pi\left(u+\tau f(t_o,u) \right)
      \right)
    \right)
    \\
    & &
    \qquad
    -
    \Pi \left( \Pi\left(u+\tau f(t_o,u) \right) 
      + 
      \tau' f(t_o, u)
    \right)
    \Big\|
    \\
    & &
    +
    \norma{
      \Pi \left( \Pi\left(u+\tau f(t_o,u) \right) 
        + 
        \tau' f(t_o, u)
      \right)
      -
      \Pi \left( u+(\tau'+\tau)f(t_o,u)\right)
    } \,.
  \end{eqnarray*}
  Consider the two summands above separately. Using~\textbf{(C)}, the
  latter one is estimated in Lemma~\ref{lem:convex}. The former is
  estimated as follows:
  \begin{eqnarray*}
    & &
    \Big\|
    \Pi \left( \Pi\left(u+\tau f(t_o,u) \right) 
      + 
      \tau' f \left(t_o+\tau, \Pi\left(u+\tau f(t_o,u) \right)
      \right)
    \right)
    \\
    & &
    \qquad
    -
    \Pi \left( \Pi\left(u+\tau f(t_o,u) \right) 
      + 
      \tau' f(t_o, u)
    \right)
    \Big\|
    \\
    & \leq & \Big\| \Pi\left(u+\tau f(t_o,u) \right) + \tau' f
    \left(t_o+\tau, \Pi\left(u+\tau f(t_o,u) \right) \right)
    \\
    & & 
    \qquad - \left( \Pi\left(u+\tau f(t_o,u) \right) + \tau'
      f(t_o, u) \right) \Big\|
    \\
    & \leq & \tau' \norma{ f \left(t_o+\tau, \Pi\left(u+\tau f(t_o,u)
        \right) \right) - f(t_o, u)}
    \\
    & \leq & 
    \tau' 
    \left( 
      \omega(\tau) + \left( L\, \sup \norma{f} \right)\, \tau
    \right)
  \end{eqnarray*}
Summing up the two bounds for the two terms we obtain
\begin{eqnarray*}
  \Delta
  & \leq &
  \tau' 
  \left( 
    \omega(\tau) + \left( L\, \sup \norma{f} \right)\, \tau
  \right)
  + 
  K \left( \sup \norma{f}\right) \tau \tau'
  \\
& \leq &
 \tau' \left( \omega(\tau) + (L+K) \left( \sup \norma{f}\right) \tau \right)
\end{eqnarray*}

Passing
  now to~\Ref{eq:Stability}, observe that
  \begin{eqnarray*}
    & &
    \norma{
      \Pi\left(u+\tau f(t_o,u)\right) -
      \Pi\left(w+\tau f(t_o,w)\right)
    }
    \leq
    \\
    & \leq &
    \norma{u - w + \tau \left( f(t_o,u) - f(t_o,w) \right)}
    \\
    & \leq &
    (1 + L \tau) \norma{u-w}
    \\
    & \leq &
    e^{L \tau} \norma{u-w}
  \end{eqnarray*}
  which directly implies~\Ref{eq:Stability}.
\end{proof}

We note that condition~\textbf{(C)} can be slightly relaxed. Indeed,
Proposition~\ref{prop:stop} essentially requires
that~\Ref{eq:StopCondition} is satisfied. This bound can be proved
also with only the $\C{1,\alpha}$ regularity of $\Pi$.

\subsection{Nonlinear Operator Splitting in a Metric Space}
\label{subsec:operator}

The construction in Theorem~\ref{thm:main} generalizes that
in~\cite{ColomboCorli5}. Indeed, consider the case of two Lipschitz
semigroups $S^1,S^2 \colon \reali^+ \times X \mapsto X$ and assume
that they \emph{commute} in the sense
of~\cite[\textbf{(C)}]{ColomboCorli5}, i.e.
\begin{equation}
  \label{eq:commute}
  d\left(S^1_t S^2_t u , S^2_t S^1_t u\right) \leq t\,\omega(t)
  \quad \mbox{ and } \quad
  \int_0^{\delta/2} \frac{\omega(t)}{t} \, dt < +\infty \,.
\end{equation}
for a suitable map $\omega \colon [0, \delta/2] \mapsto \reali$. In
this framework, let $\mathcal{D}_t = X$ for all $t \in \reali^+$, so
that $\mathcal{A} = \reali^+ \times \reali^+ \times X$. Introduce
\begin{equation}
  \label{eq:split}
  F(\tau,t_o) = S^1_\tau \, S^2_\tau \,.
\end{equation}
Note that $F$ is a local flow in the sense of
Definition~\ref{def:local}. The next lemma shows that~\Ref{eq:commute}
implies~\ref{it:first}.~in Theorem~\ref{thm:main}.

\begin{lemma}
  \label{lem:semigroup}
  Let $S^1,S^2$ be Lipschitz semigroups satisfying~\Ref{eq:commute}
  and define $F$ as in~\Ref{eq:split}.  Then,
  \begin{displaymath}
    d \left(
      F ( k \tau, t_o+\tau) \circ F(\tau,t_o) u, 
      F\left( (k+1) \tau, t_o \right)u
    \right)
    \leq
    \Lip(S^1)\, \Lip(S^2) \, k\tau \, \omega(\tau) .
  \end{displaymath}
\end{lemma}

\begin{proof}
  Compute:
  \begin{eqnarray*}
    & &
    d \left(
      F ( k \tau, t_o+\tau) \circ F(\tau,t_o) u, 
      F\left( (k+1) \tau, t_o \right)u
    \right) =
    \\
    & = &
    d \left( 
      S^1_{k\tau} S^2_{k\tau} S^1_\tau S^2_\tau u,
      S^1_{(k+1)\tau} S^2_{(k+1)\tau}u
    \right)
    \\
    & \leq &
    \Lip(S^1) \cdot d \left( 
      S^2_{k\tau} S^1_\tau S^2_\tau u,
      S^1_{\tau} S^2_{(k+1)\tau}u
    \right)\\
    & \leq &
    \Lip(S^1) \sum_{j=0}^{k-1} d \left(S^2_{(j+1)\tau}S^1_{\tau} 
      S^2_{(k-j)\tau}u, S^2_{j\tau} S^1_\tau S^2_{(k+1-j)\tau} u\right)\\
    & \leq &
    \Lip(S^1) \Lip(S^2) \sum_{j=0}^{k-1} d \left(S^2_{\tau}S^1_{\tau} 
      S^2_{(k-j)\tau}u,S^1_\tau S^2_{\tau}S^2_{(k-j)\tau} u\right)\\
    & \leq &
    \Lip(S^1) \Lip(S^2) \sum_{j=0}^{k-1}\tau\omega(\tau)\\
    & \leq & \Lip(S^1) \Lip(S^2) k\tau\omega(\tau)
  \end{eqnarray*}
  completing the proof.
\end{proof}

The following corollary of Theorem~\ref{thm:main} that
extends~\cite[Theorem~3.8]{ColomboCorli5} is now immediate.

\begin{corollary}
  \label{cor:Splitting}
  Let $S^1, S^2$ be Lipschitz semigroups satisfying~\Ref{eq:commute}.
  Suppose also that the two semigroups are Trotter stable, i.e.~for
  any $t \in [0,T]$, $n \in \naturali \setminus \{0\}$ and $u,w \in
  X$:
  \begin{equation}
    \label{eq:trotterstab}
    d \left(
      \left[S^1_{t/n}S^2_{t/n}\right]^nu,
      \left[S^1_{t/n}S^2_{t/n}\right]^nw
    \right)
    \leq
    \widetilde{C}\cdot d\left(u,w\right). 
  \end{equation}
  Then, the Euler $\epsilon$-polygonals with $F$ as in~\Ref{eq:split}
  converge to a unique product semigroup $P$ tangent to the local flow
  $F(\tau)u=S^1_\tau S^2_\tau u$.
\end{corollary}

Here we do not require the strong commutativity
condition~\cite[\textbf{(C$^*$)}]{ColomboCorli5} as
in~\cite[Theorem~3.8]{ColomboCorli5}, but only its weaker
version~\Ref{eq:commute}. Note also that the
assumption~\cite[\textbf{(S3)}]{ColomboCorli5} requiring $d(S_t u,S_t
w) \leq e^{Ct} \, d(u,w)$ is stronger then the Trotter stability
requirement~\Ref{eq:trotterstab} as shown in the following lemma.

\begin{lemma}
  Fix a positive $T$. Let $S^1,S^2$ be Lipschitz semigroups satisfying
  \begin{displaymath}
    d(S^1_t u, S^1_t w) \leq e^{Ct} d(u,w)
    \quad \mbox{ and } \quad
    d(S^2_t u, S^2_t w) \leq e^{Ct} d(u,w)
  \end{displaymath}
  for $u,w \in X$ and a fixed $C>0$. Then,~\Ref{eq:trotterstab} holds
  for $t \in [0,T]$.
\end{lemma}

For a proof, see~\cite[formula~(3.3) in
Proposition~3.2]{ColomboCorli5}.

\subsection{Trotter Formula for Linear Semigroups}
\label{sub:Trotter}

The present non linear framework recovers, under slightly different
assumptions, the convergence of Trotter formula~\cite{Trotter} in the
case of linear semigroups, see~\cite[Theorem~3]{KuhneWacker1}.

\begin{proposition}
  \label{prop:Trotter}
  Let $S^1,S^2$ be strongly continuous semigroups on a Banach space
  $X$. Assume that there exists a normed vector space $Y$ which is
  densely embedded in $X$ and invariant under both semigroups such
  that:
  \begin{enumerate}[(a)]
  \item the two semigroups are locally Lipschitz in time on $Y$,
    i.e.~there exists a compact map $K \colon Y \mapsto \reali$ such
    that for $i=1,2$
    \begin{displaymath}
      \norma{S^i_t u - S^i_{t'}u}_X
      \leq K(u) \, \modulo{t-t'} \,\mbox{ for all }u\in Y,\ t,t'\in I.
    \end{displaymath}
  \item the two semigroups are exponentially bounded on $F$ and
    locally Trotter stable on $X$ and $Y$, i.e.~there exists a
    constant $H$ such that for all $t \in I$,
    $n\in\naturali\setminus\{0\}$
    \begin{displaymath}
      \norma{S^1_t}_Y+\norma{S^1_t}_Y+\norma{ \left( S^1_{t/n} \, S^2_{t/n} 
        \right)^n}_X
      +
      \norma{ \left( S^1_{t/n} \, S^2_{t/n} \right)^n}_Y
      \leq H \,;
    \end{displaymath}
  \item the commutator condition
    \begin{displaymath}
      \frac{1}{t} \, \norma{S^2_t S^1_t u - S^1_t S^2_t u}_X
      \leq
      \omega(t) \, \norma{u}_Y \,.
    \end{displaymath}
    is satisfied for all $u \in Y$ and $t \in [0, \delta]$ with some
    $\delta > 0$, and for a suitable $\omega \colon [0,\delta] \mapsto
    \reali^+$ with $\int_0^{\delta} \frac{\omega(\tau)}{\tau} \, d\tau
    < +\infty$.
  \end{enumerate}
  \noindent Then, there exists a global semigroup $P \colon
  \left[0,+\infty \right[ \times X \to X$ such that
  \begin{enumerate}[(I)]
  \item for all $u \in Y$, there exists a constant $C_u$ such that
    \begin{displaymath}
      \frac{1}{t} \, \norma{P(t)u - S^1_t S^2_t u}_{X}
      \leq
      C_u \int_0^t \frac{\omega(\xi)}{\xi} \, d\xi \,.
    \end{displaymath}
  \item all Euler polygonals with initial data in $Y$, as defined
    in~\Ref{eq:Euler} and in~\Ref{eq:polygonal} with
    $F(t)=S^1_tS^2_t$, converge to the orbits of $P$.
  \end{enumerate}
\end{proposition}

\begin{remark}
  With respect to~\cite[Theorem~3]{KuhneWacker1}, the regularity
  assumptions on the semigroups are stronger due to~(a). On the other
  hand, (c) is weaker and we further obtain~(I) and~(II).
\end{remark}

In the theorem above, standard techniques allow to relate the
generators of $S^1$ and $S^2$ with that of $P$,
see~\cite[Proposition~4.1]{Chernoff} and
\cite[Theorem~3]{KuhneWacker1}.

\begin{proofof}{Proposition~\ref{prop:Trotter}}
  Fix a positive $M$ and define
  \begin{displaymath}
    \mathcal{D}_M
    =
    \left\{ u\in Y\colon \norma{u}_Y\leq M\right\}
    \quad \mbox{ and }\quad
    \mathcal{D}_M^*
    =
    \left\{ u\in Y\colon \norma{u}_Y\leq \frac M {H^6}\right\} \,.
  \end{displaymath}
  In the setting of Theorem~\ref{thm:main}, consider the metric space
  $\left(X,\norma{\cdot}_X\right)$ and define
  \begin{displaymath}
    F(t)u 
    =
    S^1_t S^2_t u
    \quad \mbox{ and }\quad
    \mathcal{D} 
    =
    {\mathrm{cl}} \mathcal{D}_M
  \end{displaymath}
  where the closure is meant with respect to the $X$ norm. It is
  straightforward to show that $F$ is a Lipschitz local flow on
  $\mathcal{D}$.  The stability (with respect to the $X$ norm) of the
  Euler $\epsilon$--polygonals implies hypothesis 2.  Computations
  similar to those in the proof of Lemma~\ref{lem:semigroup} show that
  also hypothesis 1. is satisfied. Hence, Theorem~\ref{thm:main}
  yields the existence of a process tangent to the local flow $F$. The
  stability (in the $Y$ norm) of the Euler $\epsilon$--polygonals
  ensures that ${\mathrm{cl}} \mathcal{D}^*_M \subseteq
  \mathcal{D}^3_{t_o}$ for any $t_o\in I$.  Finally, the arbitrariness
  of $M$ allows us to extend $P$ to all $Y$ and, by density, to all
  $X$.
\end{proofof}

\subsection{Hille-Yosida Theorem}

This section is devoted to show that Theorem~\ref{thm:main}
comprehends the constructive part of Hille-Yosida Theorem,
see~\cite[Chapter~II, Theorem~3.5]{EngelNagel}
or~\cite[Theorem~12.3.1]{HillePhillips}. The resulting proof, which we
sketch below, is similar to the original proof by Hille,
see~\cite[p.~362]{HillePhillips}.

\begin{theorem}
  \label{thm:HY}
  Let $X$ be a Banach space and $A$ be a linear operator with domain
  $D(A)$. If $\left( A, D(A)\right)$ is closed, densely defined and
  for all $\lambda > 0$ one has $\lambda \in \rho(A)$ and
  $\norma{\lambda \, R(\lambda,A)} \leq 1$, then $A$ generates a
  strongly continuous contraction semigroup.
\end{theorem}

Above, as usual, $\rho(A)$ is the resolvent set of $A$ and
$R(\lambda,a)$ is the resolvent. Note that we consider only the
contractive case and refer to~\cite[Theorem~3.8,
Chapter~II]{EngelNagel} to see how the general case can be recovered.

\begin{proofof}{Theorem~\ref{thm:HY}}
  Fix a positive $M$ and define
  \begin{eqnarray*}
    \mathcal{D}_t = \mathcal{D}
    & = & 
    {\mathrm{cl}} \left\{ 
      u \in D\left(A^2\right) \colon \ 
      \norma{A u} \leq M \mbox{ and } \norma{A^2u} \leq M 
    \right\}
    \\
    F(t) u 
    & = &
    \left\{
      \begin{array}{l@{\qquad}rcl}
        \displaystyle \frac{1}{t}\, R\left( \frac{1}{t}, A \right) u 
        & t & > & 0 \,,
        \\
        u & t & = & 0 \,.
      \end{array}
    \right.
  \end{eqnarray*}
  It is easy to see that $F(t) \mathcal{D} \subseteq \mathcal{D}$, so
  that~\textbf{(D)} holds. Moreover, the stability
  condition~\ref{it:second}.~in Theorem~\ref{thm:main} trivially holds
  in the contractive case.  Recall the usual identities for the
  resolvent (see~\cite[Chapter~IV]{EngelNagel})
  \begin{eqnarray*}
    & & 
    \lambda \, R(\lambda,A) u = u + R(\lambda, A) \, A\, u\quad
    \mbox{ for any }u\in D(A)
    \\
    & & 
    R(\lambda,A)-R(\mu,A)=\left(\mu-\lambda\right)R(\lambda,A)R(\mu,A).
  \end{eqnarray*}
  Written using the local flow $F$, these identities become
  \begin{eqnarray*}
    & & 
    F(t) u = u + tF(t)Au\quad  \mbox{ for any }u\in D(A)
    \\
    & & 
    tF(t)-sF(s)=(t-s)F(t) F(s) \, .
  \end{eqnarray*}
  Now we first show that $F$ is Lipschitz in $t$ and $u$. The
  Lipschitz continuity with respect to $u$ is a straightforward
  consequence of the bound on the resolvent norm. Concerning the
  variable $t$, we take $u\in D(A)$ with $\norma{Au} \leq M$ and
  consider two cases:
  \begin{displaymath}
    \begin{array}{rclcl}
      \norma{F(t)u-F(0)u}
      & = &
      \norma{tF(t)Au}
      &\leq& 
      t M
      \\
      \norma{F(t)u-F(s)u}
      & = &
      \norma{tF(t)Au-sF(s)Au} 
      &\leq&
      \norma{(t-s)F(t)F(s)Au}
      \\
      &\leq& 
      \modulo{t-s} M \,.
    \end{array}
  \end{displaymath}
  Since $F(t)$ is a bounded operator, the Lipschitz continuity with
  constant $M$ extends to all $\mathcal{D}$.  We are left to prove
  hypothesis~\ref{it:first}.~in Theorem~\ref{thm:main}. For $u\in
  D(A^2)$ with $\norma{A^2u}\leq M$, $\norma{Au}\leq M$, compute
  \begin{eqnarray*}
    & &
    \norma{F(t) F(s) u - F(t+s) u} =
    \\
    & = &
    \norma{
      \left(\Id + t F(t)A\right) \left(\Id + s F(s)A\right)u
      -
      \left(\Id + (t+s) F(t+s)A\right)u
    }
    \\
    & = &
    \norma{
      t F(t)A u +  s F(s)A u + t s F(t) F(s) A^2 u  - (t+s) F(t+s) A u
    }
    \\
    & \leq &
    t s \norma{A^2 x} 
    +
    t\norma{
      \left(F(t)-F(t+s)\right)A u}
    +
    s\norma{\left(F(s)-F(t+s)\right)A u}
    \\
    & \leq&
    t s M 
    +
    ts\norma{A^2 u}
    +
    st\norma{A^2 u}
    \\
    & \leq &
    3 t s M \,.
  \end{eqnarray*}
  Again, the boundedness of $F(t)$ allows us to extend the inequality
  \begin{displaymath}
    \norma{F(t) F(s) u - F(t+s) u} \leq 3stM
  \end{displaymath}
  to all $u\in\mathcal{D}$.  Therefore, letting $t=k\tau$, $s=\tau$,
  Theorem~\ref{thm:main} applies with $\omega(\tau) = 3M\tau$. Due to
  the arbitrariness of $M$ and the density of $D(A^2)$, the resulting
  semigroup can be extended to all $X$. Standard computations,
  see~\cite[p.~362--363]{HillePhillips}, show that $A$ is the
  corresponding generator.
\end{proofof}

Note that, by Proposition~\ref{prop:Euler}, we also provide the
convergence of all polygonal approximation.

\subsection{The Heat Equation}

Let $T$, $\delta$ be positive and $X = \CB0 (\reali; \reali)$ be the
set of continuous and bounded real functions defined on $\reali$
equipped with the distance $d (u,w) = \norma{u-w}_{\C0}$, where
$\norma{u}_{\C0} = \sup_{x \in \reali} \norma{u(x)}$. Fix a positive
$M$ and for all $t_o \geq 0$ let $\mathcal{D}_{t_o} = \mathcal{D}$ be
the subset of $X$ consisting of all twice differentiable functions
whose second derivative $u''$ satisfies $\max\left\{ \norma{u''},
  \Lip(u'') \right\} \leq M$. For $t \in [0,\delta]$, $t_o \in \reali$
and $u \in X$, using the numerical algorithm~\cite[Chapter~9,
\S~4]{IsaacsonKeller}, define
\begin{displaymath}
  \left( F(t) u \right) (x)
  =
  u(x) +
  \frac{u \left(x-2\sqrt{t}\right) - 2 u(x) +u \left(x+2\sqrt{t}\right)}{4}
\end{displaymath}

\begin{proposition}
  $F$ is a local flow satisfying~\Ref{eq:k} and~\Ref{eq:Stability}.
\end{proposition}

\begin{proof}
  Note that $u \mapsto F(t)u$ is linear.  For $u \in \mathcal{D}$,
  introduce
  \begin{displaymath}
    (D^2_\sigma u) \, (x)
    = 
    \frac{u (x-\sigma) - 2 u(x) + u (x+\sigma)}{\sigma^2} \,.
  \end{displaymath}
  $D^2_\sigma$ is a linear operator and $F(t) u = u + t
  D^2_{2\sqrt{t}}u$. Moreover,
  \begin{displaymath}
    (D^2_\sigma u) \, (x)
    =
    \int_0^1 \left(\int_{-1}^{+1} \xi \, 
      u''( x+ \eta\xi\sigma) \, d\eta \right)\, d\xi
  \end{displaymath}
  so that
  \begin{eqnarray*}
    \norma{D^2_{\sigma_1} u - D^2_{\sigma_2} u}_{\C0}
    & \leq &
    \Lip(u'') \, \modulo{\sigma_1 - \sigma_2}
    \\
    \norma{D^2_\sigma u}_{\C0}
    & \leq &
    \min 
    \left\{
      \frac{4}{\sigma^2}\norma{u}_{\C0},\norma{u''}_{\C0},
      \frac{1}{3} \, \Lip(u'') \, \sigma
    \right\} \,.
  \end{eqnarray*}
  $F$ is Lipschitz in $t$, indeed if $t_1 < t_2$, then
  \begin{eqnarray*}
    \norma{F(t_2)u - F(t_1)u}_{\C0}
    & = &
    \norma{t_2 \, D^2_{2\sqrt{t_2}} u - t_1 \, D^2_{2\sqrt{t_1}}u}_{\C0}
    \\
    & \leq &
    \modulo{t_2 - t_1} \, \norma{u''}_{\C0}
    +
    2 \, \Lip(u'') \, t_1 \, \modulo{\sqrt{t_2} - \sqrt{t_1}} 
    \\
    & \leq &
    \left( \norma{u''}_{\C0} + \Lip(u'') \sqrt{\delta} \right)
    \modulo{t_2-t_1}\,.
  \end{eqnarray*}

  Consider now~\ref{it:first}.~and write $s=kt$ for $k \in \naturali$.
  Then
  \begin{eqnarray*}
    & &
    \norma{F(kt)F(t) u - F(kt+t)u}_{\C0} =
    \\
    & = &
    \norma{
      F(kt) \left(u + t D^2_{2\sqrt{t}} u \right) -
      u -
      (t+kt) D^2_{2\sqrt{t+kt}}u
    }_{\C0}
    \\
    & = &
    \norma{
      t D^2_{2\sqrt{t}} u + 
      kt D^2_{2\sqrt{kt}} u + 
      kt t D^2_{2\sqrt{kt}} D^2_{2\sqrt{t}} u -
      (t+kt) D^2_{2\sqrt{t+kt}} u
    }_{\C0}
    \\
    & \leq &
    t \norma{D^2_{2\sqrt{t}} u - D^2_{2\sqrt{t+kt}} u}_{\C0} +
    kt \norma{D^2_{2\sqrt{kt}} u - D^2_{2\sqrt{t+kt}} u}_{\C0}\\
    & & +
    kt^2 \norma{D^2_{2\sqrt{kt}} D^2_{2\sqrt{t}} u}_{\C0}
    \\
    & \leq &
    2\Lip(u'') \left(
      t \modulo{\sqrt{t+kt} - \sqrt{t}} +
      kt \modulo{\sqrt{t+kt} - \sqrt{kt}}
    \right)
    +
    t \norma{D^2_{2\sqrt{t}}u}_{\C0}
    \\
    & \leq &
    2\, \Lip(u'') \, \left( t\sqrt{kt} + kt \sqrt{t} \right) +
    \frac{2}{3} \, \Lip(u'') \, t \sqrt{t}
    \\
    & \leq &
    \frac{14}{3} \, M \, kt \sqrt{t}
  \end{eqnarray*}
  hence, \ref{it:first}.~in Theorem~\ref{thm:main} is satisfied with
  $\omega(t) = \frac{14}{3} \, M \, \sqrt{t}$.

  Finally, the equality
  \begin{displaymath}
    \left( F(t) u \right) (x)
    = 
    \frac{1}{4} \, u\left(x-2\sqrt{t} \right)
    +
    \frac{1}{2} \, u(x)
    + 
    \frac{1}{4} \, u\left(x+2\sqrt{t} \right)
  \end{displaymath}
  implies that $\left( F(t) u \right) (\reali) \subseteq {\mathop{\rm
      co}} \left( u(\reali)\right)$. Hence, for all $u,w \in X$,
  \begin{eqnarray*}
    d \left( F(t)u, F(t)w \right)
    =
    \norma{ F(t) (w-u)}_{\C0}
    \leq
    \norma{w-u}_{\C0}
    = d(u,w)
  \end{eqnarray*}
  showing that $F$ is non expansive in $u$ and, hence,
  that~\ref{it:second}.~holds.
\end{proof}

\Section{Technical Details}
\label{sec:Tech}

Occasionally, for typographical reasons, we write $d\Bigl( \!\!
\begin{array}{c} u,\\w \end{array} \!\! \Bigr)$ for $d(u,w)$.

In this section, we use the following definition
\begin{equation}
  \label{def:D2}
  \!\!\!
  \mathcal{D}^2_{t_o}
  =
  \left\{
    u\in\mathcal{D} \colon \!
    \begin{array}{l}
      F^{\epsilon_2} (t_2,t_o+t_1) \circ 
      F^{\epsilon_1} (t_1,t_o)u 
      \\
      \mbox{is in } \mathcal{D}
      \mbox{ for all }
      \epsilon_1, \epsilon_2 \in
      \left]0,\delta \right]
      \mbox{ and all} 
      \\ 
      t_1, t_2 \geq 0 \mbox{ such that } t_o+t_1+t_2 \in I
    \end{array}\!
  \right\}
\end{equation}
Observe that one obviously has
$\mathcal{D}^3_{t_o}\subseteq\mathcal{D}^2_{t_o}$.

\begin{proofof}{Lemma~\ref{lem:Time}}
  $F^\epsilon(t,t_o)u$ is continuous in $t_o$ and Lipschitz in $u$
  because it is the composition~\Ref{eq:polygonal} of functions with
  the same properties, see Definition~\ref{def:local}.  Fix now
  positive $s,t$ with $t_o\leq s \leq t \leq T$. Then note that if
  $h\epsilon \leq s \leq t \leq (h+1)\epsilon$ for $h \in \naturali$,
  then by Definition~\ref{def:local},
  \begin{displaymath}
    d\left( F^\epsilon (s,t_o)u, F^\epsilon(t,t_o)u\right)
    \leq
    \Lip(F)  \, \modulo{s-t} \,.
  \end{displaymath}
  Let $k = [t/\epsilon]$ and $h = [s/\epsilon]+1$. The case $h > k$ is
  recovered by the previous computation.  If $h \leq k$, then
  by~\Ref{eq:polygonal}
  \begin{eqnarray*}
    d \left( F^\epsilon(s,t_o) u, F^\epsilon (t,t_o) u \right)
    & \leq &
    d \left( 
      F^\epsilon(s, t_o) u , F^\epsilon(h\epsilon, t_o) u
    \right)
    \\
    & &
    +
    \sum_{j=h}^{k-1}
    d \left( 
      F^\epsilon(j\epsilon, t_o) u, 
      F^\epsilon\left( (j+1)\epsilon, t_o \right) u 
    \right)
    \\
    & &
    +
    d \left(
      F^\epsilon (k\epsilon,t_o) u, F^\epsilon (t,t_o)u
    \right)
    \\
    & \leq &
    \Lip(F) \, 
    \left(
      (h\epsilon - s) + (k-h)\epsilon + (t-k\epsilon)
    \right)
    \\
    & = &
    \Lip(F) \, (t-s) \,.
  \end{eqnarray*}
\end{proofof}

\begin{lemma}
  \label{lem:kappa}
  Let $F$ be a local flow satisfying~1.~in
  Theorem~\ref{thm:main}. Then, for all $h,k,t_o,u$ and $\epsilon$
  satisfying $t_o \in I$, $h,k \in \naturali$, $k\epsilon \in \left]0,
    \delta \right]$, $t_o+(h+k)\epsilon \in I$, $u \in
  \mathcal{D}^2_{t_o}$, the following holds:
  \begin{displaymath}
    d \left( F(k\epsilon,\bar t)v, F^\epsilon (k\epsilon,\bar t)v \right)
    \leq
    k^2 \epsilon \, \omega(\epsilon) \,.
  \end{displaymath}
  where $v = F^\epsilon (h\epsilon,t_o)u$ and $\bar t =
  t_o+h\epsilon$.
\end{lemma}

\begin{proof}
  The Definition~\Ref{def:D2} of $\mathcal{D}^2_{t_o}$ implies that we
  can apply the triangle inequality and hypothesis~1.~in
  Theorem~\ref{thm:main} to obtain
  \begin{eqnarray*}
    & &
    d \left( 
      F(k\epsilon,\bar t)v, 
      F^\epsilon (k\epsilon,\bar t)v
    \right)
    \leq
    \\
    & \leq &
    \sum_{j=1}^{k-1} 
    d \left(
      \begin{array}{l}
        F \left( 
          (j+1)\epsilon, \left(k-(j+1)\right)\epsilon+\bar t
        \right)
        \circ
        \\
        \qquad\qquad
        \circ
        F^\epsilon \left( \left(k-(j+1)\right) \epsilon, \bar t \right) v
        ,
        \\
        F \left( j\epsilon, (k-j)\epsilon+\bar t\right)
        \circ
        F \left(\epsilon,\left(k-(j+1)\right)\epsilon+\bar t\right)
        \circ
        \\
        \qquad\qquad
        \circ
        F^\epsilon \left(\left(k-(j+1)\right) \epsilon, \bar t\right)v
      \end{array}
    \right)
    \\
    & \leq &
    \sum_{j=1}^{k-1} j \, \epsilon \, \omega(\epsilon)
    \\
    & \leq &
    k^2 \, \epsilon \, \omega(\epsilon)
  \end{eqnarray*}
  completing the proof.
\end{proof}

\begin{lemma}
  \label{lem:hk}
  Let the local flow $F$ satisfy the assumptions of
  Theorem~\ref{thm:main}. Then, for all $h,k,t_o,u$ and $\epsilon$
  satisfying $t_o \in I$, $h,k \in \naturali$, $k\epsilon \in \left]0,
    \delta \right]$, $t_o+hk\epsilon \in I$ and $u \in
  \mathcal{D}^2_{t_o}$, the following holds:
  \begin{displaymath}
    d\left(
      F^{k\epsilon} (h k \epsilon, t_o)u, 
      F^\epsilon(h k \epsilon, t_o ) u
    \right)
    \leq
    Lh \, k^2 \, \epsilon \, \omega(\epsilon) \,.
  \end{displaymath}
\end{lemma}

\begin{proof}
  The Definition~\Ref{def:D2} of $\mathcal{D}^2_{t_o}$ implies that we
  can apply the triangle inequality. Then, by the assumptions in
  Theorem~\ref{thm:main}, applying Lemma~\ref{lem:kappa} and with
  computations similar to the ones in~\cite{KuhneWacker1}, we have
  \begin{eqnarray*}
    & &
    d\left(
      F^{k\epsilon} (h k \epsilon, t_o)u, 
      F^\epsilon(h k \epsilon, t_o ) u
    \right)
    \leq
    \\
    & \leq &
    \sum_{j=0}^{h-1}
    d\left(
      \begin{array}{l}
        F^{k\epsilon} 
        \left( (j+1)k \epsilon, \left(h-(j+1)\right)k \epsilon + t_o \right)
        \circ
        \\
        \qquad
        \circ
        F^\epsilon \left( \left(h-(j+1) \right) k\epsilon,t_o\right) u,
        \\
        F^{k\epsilon} \left(jk\epsilon, (h-j) k \epsilon + t_o\right)
        \circ
        F^\epsilon \left( (h-j) k \epsilon, t_o\right)u
      \end{array}
    \right)
    \\
    & \leq &
    \sum_{j=0}^{h-1}
    d\left(
      \begin{array}{l}
        F^{k\epsilon} 
        \left( j k \epsilon, (h-j)k \epsilon + t_o \right)
        \circ
        F^{k\epsilon} 
        \left( k \epsilon, \left(h-(j+1)\right)k\epsilon + t_o \right)
        \\
        \qquad
        \circ
        F^\epsilon \left( \left(h-(j+1) \right) k\epsilon,t_o\right) u,
        \\
        F^{k\epsilon} \left(jk\epsilon, (h-j) k \epsilon + t_o\right)
        \circ
        F^\epsilon 
        \left( 
          k \epsilon, \left(h-(j+1)\right) k \epsilon + t_o
        \right)
        \\
        \qquad
        \circ
        F^\epsilon \left( \left(h-(j+1)\right)k \epsilon, t_o\right) u
      \end{array}
    \right)
    \\
    & = &
    L
    \sum_{j=0}^{h-1}
    d\left(
      \begin{array}{l}
        F^{k\epsilon} 
        \left( k \epsilon, \left(h-(j+1)\right)k\epsilon + t_o \right)
        \\
        \qquad
        \circ
        F^\epsilon \left( \left(h-(j+1) \right) k\epsilon,t_o\right) u,
        \\
        F^\epsilon 
        \left( 
          k \epsilon, \left(h-(j+1)\right) k \epsilon + t_o
        \right)
        \\
        \qquad
        \circ
        F^\epsilon \left( \left(h-(j+1)\right)k \epsilon, t_o\right) u
      \end{array}
    \right)
    \\
    & = &
    L
    \sum_{j=0}^{h-1}
    d\left(
      \begin{array}{l}
        F
        \left( k \epsilon, \left(h-(j+1)\right)k\epsilon + t_o \right)
        \\
        \qquad
        \circ
        F^\epsilon \left( \left(h-(j+1) \right) k\epsilon,t_o\right) u,
        \\
        F^\epsilon 
        \left( 
          k \epsilon, \left(h-(j+1)\right) k \epsilon + t_o
        \right)
        \\
        \qquad
        \circ
        F^\epsilon \left( \left(h-(j+1)\right)k \epsilon, t_o\right) u
      \end{array}
    \right)
    \\
    & \leq &
    L \sum_{j=0}^{h-1} k^2 \epsilon \, \omega(\epsilon)
    \\
    & = &
    L \, h k^2 \epsilon \, \omega(\epsilon)
  \end{eqnarray*}
  completing the proof.
\end{proof}

\begin{lemma}
  \label{lem:integral}
  Let the local flow $F$ satisfy the assumptions of
  Theorem~\ref{thm:main}. Then, for all $m,n,t,t_o$ and $u$ satisfying
  $t>0$, $t_o, t_o+t \in I$, $m,n \in \naturali$, $n \geq m$, $t
  2^{-m} \in \left]0, \delta \right]$ and $u \in \mathcal{D}^2_{t_o}$,
  the following holds:
  \begin{displaymath}
    d\left(F^{t2^{-m}} (t,t_o) u, F^{t2^{-n}} (t,t_o)u \right)
    \leq 
    \frac{2L}{\ln 2} \, t \,
    \int_{t2^{-n}}^{t2^{-m}} \frac{\omega(\xi)}{\xi} \, d\xi \,.
  \end{displaymath}
\end{lemma}

\begin{proof}
  Applying Lemma~\ref{lem:hk} with $\epsilon = t 2^{-(j+1)}$, $k=2$,
  $h=2^j$ and $t=h k \epsilon$,
  \begin{eqnarray*}
    & &
    d\left( F^{t2^{-m}} (t,t_o) u, F^{t2^{-n}} (t,t_o)u \right)
    \leq
    \\
    & \leq &
    \sum_{j=m}^{n-1}
    d\left(
      F^{2t2^{-(j+1)}} (t, t_o) u, F^{t2^{-(j+1)}} (t, t_o) u
    \right)
    \\
    & \leq &
    L \sum_{j=m}^{n-1} 2 t \, \omega(t2^{-(j+1)})
    \\
    & \leq &
    2 L t \sum_{j=m}^{n-1} \int_j^{j+1} \omega(t2^{-(j+1)}) \, ds
    \\
    & \leq &
    \frac{2Lt}{\ln 2} 
    \int_{t2^{-n}}^{t2^{-m}} \frac{\omega(\xi)}{\xi} \, d\xi \,.
  \end{eqnarray*}
\end{proof}

\begin{corollary}
  \label{cor:Pi}
  Let $F$ be a local flow satisfying the assumptions of
  Theorem~\ref{thm:main}. Then, for any $t>0$, $t_o,t_o+t\in I$,
  $u\in\mathcal{D}^2_{t_o}$ and $T2^{-m} \in \left]0, \delta \right]$,
  the sequence of functions
  \begin{displaymath}
    (t,t_o,u) \mapsto\left\{
      \begin{array}{lll}
        F^{t2^{-m}} (t,t_o) u &\mbox{ for }&t>0\\
        u&\mbox{ for }&t=0
      \end{array}\right.
  \end{displaymath}
  converges uniformly to a continuous map $(t,t_o,u) \to P (t,t_o)u$.
\end{corollary}

\begin{proof}
  Observe that the maps in the sequence above are continuous, since
  $F^{t2^{-m}}(t,t_o)u = \comp_{j=0}^{2^m-1} F \left( t2^{-m},
    t_o+j2^{-m}t\right)u$ is the composition of $2^m$ continuous
  maps. The uniform convergence follows from Lemma~\ref{lem:integral}.
\end{proof}

\begin{remark}
  \label{rem:basta2}
  To prove the existence of a limit $P$, i.e.~Lemma~\ref{lem:integral}
  and Corollary~\ref{cor:Pi}, we need Lemma~\ref{lem:hk} and hence
  Lemma~\ref{lem:kappa} only for $k=2$. On the other hand, to prove
  these lemmas for $k=2$, it is enough to assume 1.~of
  Theorem~\ref{thm:main} only for $k=1$. Assumption 1. in
  Theorem~\ref{thm:main} for all $k\in\naturali$ is needed to show
  that all the sequence of Euler approximates $F^\epsilon$ converges
  (Proposition~\ref{prop:45}) and, hence, that the limit is indeed a
  process (Lemma~\ref{lemma4.7}). The necessity of hypothesis~1.~for
  all $k\in\naturali$ is shown in Paragraph~\ref{sub:example}
\end{remark}

\begin{proposition}
  \label{prop:45}
  Let $F$ be a local flow satisfying the assumptions of
  Theorem~\ref{thm:main}. Then, for all $t_o,t,u$ with $t\geq 0$,
  $t_o,t_o+t \in I$ and $u \in \mathcal{D}^2_{t_o}$, the following
  limit holds:
  \begin{displaymath}
    \lim_{\epsilon \to 0+} F^\epsilon(t,t_o)u = P(t,t_o)u \,,
  \end{displaymath}
  with $P$ defined as in Corollary~\ref{cor:Pi}.
\end{proposition}

\begin{proof}
  Fix $t_o,t, u$ as above. Let $n_\epsilon =[t/\epsilon]$. Observe
  first that
  \begin{eqnarray*}
    d \left( F^\epsilon (t,t_o)u, P(t,t_o)u \right)
    & \leq &
    d 
    \left( 
      F^\epsilon (t,t_o)u, F^\epsilon(n_\epsilon \epsilon,t_o)u
    \right)
    \\
    & &
    +
    d 
    \left( 
      F^\epsilon(n_\epsilon \epsilon,t_o)u, P(n_\epsilon \epsilon,t_o)u
    \right)
    \\
    & &
    +
    d \left( P(n_\epsilon \epsilon,t_o)u, P(t,t_o)u \right)
    \\
    & \leq &
    d 
    \left( 
      F^\epsilon(n_\epsilon \epsilon,t_o)u, P(n_\epsilon \epsilon,t_o)u
    \right)
    + o(1)
  \end{eqnarray*}
  as $\epsilon \to 0$, due to the Lipschitz continuity of $F$ and the
  continuity of $P$. With computations similar to the ones found
  in~\cite{KuhneWacker1}, for $l\in\naturali$ such that $T2^{-l} <
  \delta$ write:
  \begin{equation}
    \label{eq:3terms}
    \begin{array}{cl}
      &
      d 
      \left( 
        F^\epsilon(n_\epsilon \epsilon,t_o)u, P(n_\epsilon \epsilon,t_o)u
      \right)
      \leq
      \\
      \leq &
      d 
      \left( 
        F^\epsilon(n_\epsilon \epsilon,t_o)u, 
        F^{\epsilon 2^{-l}}(n_\epsilon \epsilon,t_o)u
      \right)
      \\
      &
      +
      d 
      \left( 
        F^{\epsilon 2^{-l}}(n_\epsilon \epsilon,t_o)u,
        F^{n_\epsilon \epsilon 2^{-l}}(n_\epsilon \epsilon,t_o)u,
      \right)
      \\
      &
      +
      d 
      \left( 
        F^{n_\epsilon \epsilon 2^{-l}}(n_\epsilon \epsilon,t_o)u, 
        P(n_\epsilon \epsilon,t_o)u
      \right)
    \end{array}
  \end{equation}
  Consider the three terms separately. Estimating the former one,
  apply first Lemma~\ref{lem:hk} with $h=n_\epsilon 2^{j}$, $k=2$ and
  $\epsilon$ substituted by $\epsilon 2^{-(l+1)}$ and then treat the
  summation as in the proof of Lemma~\ref{lem:integral}:
  \begin{eqnarray*}
    & &
    d 
    \left( 
      F^\epsilon(n_\epsilon \epsilon,t_o)u, 
      F^{\epsilon 2^{-l}}(n_\epsilon \epsilon,t_o)u
    \right)
    \leq
    \\
    & \leq &
    \sum_{j=0}^{l-1}
    d\left(
      F^{2\epsilon2^{-(j+1)}} (n_\epsilon \epsilon, t_o)u,
      F^{\epsilon 2^{-(j+1)}} (n_\epsilon \epsilon, t_o)u
    \right)
    \\
    & \leq &
    L \sum_{j=0}^{l-1}
    2 n_\epsilon \epsilon \omega(\epsilon2^{-(j+1)})
    \\
    & \leq &
    \frac{2Lt}{\ln 2} \, 
    \int_0^\epsilon \frac{\omega(\xi)}{\xi} \, d\xi
    \\
    & \to &
    0
    \mbox{ as } \epsilon \to 0+ \mbox{ uniformly in }l.
  \end{eqnarray*}
  Estimate the second term in~\Ref{eq:3terms} again using
  Lemma~\ref{lem:hk} with $h=2^{l}$, $k=n_\epsilon$ and $\epsilon$
  substituted by $\epsilon 2^{-l}$:
  \begin{eqnarray*}
    d 
    \left( 
      F^{\epsilon 2^{-l}}(n_\epsilon \epsilon,t_o)u,
      F^{n_\epsilon \epsilon 2^{-l}}(n_\epsilon \epsilon,t_o)u,
    \right)
    & \leq &
    L {n_\epsilon}^2 \, \epsilon \, \omega(\epsilon 2^{-l})
    \\
    & \to &
    0
    \mbox{ as } l \to +\infty \mbox{ for fixed } \epsilon \,.
  \end{eqnarray*}
  Finally, also the latter term in~\Ref{eq:3terms} vanishes as $l \to
  +\infty$ by Corollary~\ref{cor:Pi}.

  The proof is completed passing in~\Ref{eq:3terms} to the limits
  first $l \to +\infty$ and, secondly, $\epsilon \to 0$.
\end{proof}

\begin{remark}
  \label{remark4.6}
  Observe that, since $F^\epsilon(t,t_o)u$ is uniformly Lipschitz in
  $(t,u)$, $P(t,t_o)u$ is also uniformly Lipschitz in $(t,u)$ with the
  same constant as $F^\epsilon$.
\end{remark}

\begin{lemma}\label{lemma4.7}
  There exists a family of sets
  $\mathcal{D}_{t_o}\subseteq\mathcal{D}$, for $t_o\in I$ such that
  \begin{enumerate}[i)]
  \item $\mathcal{D}^3_{t_o} \subseteq \mathcal{D}_{t_o} \subseteq
    \mathcal{D}^2_{t_o}$ for any $t_o\in I$;
  \item The map $P$ defined in Corollary~\ref{cor:Pi} and restricted
    to the set $\mathcal{A} = \left\{(t,t_o,u) \colon \ t\geq 0,\
      t_o,t_o+t\in I,\ u\in\mathcal{D}_{t_o}\right\}$ is a global
    process according to Definition~\ref{def:Local} and Lipschitz in
    all its variables.
  \end{enumerate}
\end{lemma}

\begin{proof}
  Define
  \begin{displaymath}
    \mathcal{D}_{t_o}
    =
    \left\{u\in\mathcal{D}^2_{t_o}:\
      P(t,t_o)u\in\mathcal{D}^2_{t_o+t}\ 
      \mbox{ for any } t\geq 0
      \mbox{ such that }t_o+t\in I\right\}.
  \end{displaymath}
  Obviously, $\mathcal{D}_{t_o}\subseteq\mathcal{D}^2_{t_o}$. Take now
  $u\in\mathcal{D}^3_{t_o}$, we want to show that $u$ also belongs to
  $\mathcal{D}_{t_o}$. By~\Ref{def:D3}
  \begin{displaymath}
    F^{\epsilon_2} (t_2,t_o+t+t_1)\circ F^{\epsilon_1} 
    (t_1,t_o+t)\circ F^{\epsilon} 
    (t,t_o)u \in \mathcal{D}
  \end{displaymath}
  for any $\epsilon,\epsilon_1,\epsilon_2 \in \left]0, \delta \right]$
  and $t,t_1,t_2\geq 0$ such that $t_o+t+t_1+t_2\in I$. Since $u\in
  \mathcal{D}^2_{t_o}$, $\mathcal{D}$ is close and $F^{\epsilon_1}$,
  $F^{\epsilon_2}$ are Lipschitz, we let $\epsilon \to 0$ and obtain
  \begin{displaymath}
    F^{\epsilon_2} (t_2,t_o+t+t_1)\circ F^{\epsilon_1} 
    (t_1,t_o+t)\circ P 
    (t,t_o)u\in\mathcal{D}
  \end{displaymath}
  for any $\epsilon_1,\epsilon_2\in \left]0, \delta \right]$ and
  $t_1,t_2\geq 0$ such that $t_o+t+t_2+t_3\in I$, that is
  $P(t,t_o)u\in \mathcal{D}^2_{t_o}$ for all $t\geq 0$ such that
  $t_o+t\in I$ and hence $u\in\mathcal{D}_{t_o}$.  For $t_o\in I$ and
  $u\in\mathcal{D}^2_{t_o}\supset\mathcal{D}_{t_o}$ one trivially has
  $P(0,t_o)u=u$.  We are left to prove the ``semigroup
  properties''~\Ref{eq:process0} and~\Ref{eq:process}. We first
  show~\Ref{eq:process} for any $u \in \mathcal{D}_{t_o}$. If $t_1$ or
  $t_2$ vanishes, property~\Ref{eq:process} is trivial. Fix $t_1>0$
  and take $t_2>0$ such that $t_o+t_1+t_2\in I$ and
  $\frac{t_2}{t_1}\in\mathbb{Q}$ so that there exist two integers
  $h,k\in\naturali\setminus \{0\}$ which satisfy
  $\frac{t_1}{k}=\frac{t_2}{h}$. For any $\nu\in\naturali\setminus
  \{0\}$ define
  $\epsilon_\nu=\frac{1}{\nu}\frac{t_1}{k}=\frac{1}{\nu}\frac{t_2}{h}$. For
  $\epsilon_\nu < \delta$, by Definition~\ref{eq:polygonal} we have
  \begin{equation}
    \label{eq:nulimit}
    F^{\epsilon_\nu}(t_2,t_o+t_1)\circ F^{\epsilon_\nu}(t_1,t_o)u
    =
    F^{\epsilon_\nu}(t_2+t_1,t_o)u
  \end{equation}
  Since $u\in\mathcal{D}_{t_o}\subseteq\mathcal{D}^2_{t_o}$ one has
  that $F^{\epsilon_\nu}(t_2+t_1,t_o)u$ converges to $P(t_2+t_1,t_o)u$
  as $\nu\to +\infty$ (see Proposition~\ref{prop:45}).  Moreover the
  definition of $\mathcal{D}_{t_o}$ implies that
  $P(t_1,t_o)u\in\mathcal{D}^2_{t_o+t_1}$ and therefore
  $F^{\epsilon_\nu}(t_2,t_o+t_1)\circ P(t_1,t_o)u$ converges to
  $P(t_2,t_o+t_1)\circ P(t_1,t_o)u$. We can so compute
  \begin{eqnarray*}
    & &
    d\left(
      F^{\epsilon_\nu}(t_2,t_o+t_1) \circ F^{\epsilon_\nu}(t_1,t_o)u,
      P(t_2,t_o+t_1) \circ P(t_1,t_o)u
    \right)
    \leq
    \\
    & \leq &
    d\left(
      F^{\epsilon_\nu}(t_2,t_o+t_1)\circ F^{\epsilon_\nu}(t_1,t_o)u,
      F^{\epsilon_\nu}(t_2,t_o+t_1)\circ P(t_1,t_o)u
    \right)
    \\
    & &
    +
    d\left(
      F^{\epsilon_\nu}(t_2,t_o+t_1)\circ P(t_1,t_o)u,
      P(t_2,t_o+t_1)\circ P(t_1,t_o)u
    \right)
    \\
    & \leq &
    L d \left(F^{\epsilon_\nu}(t_1,t_o)u,P(t_1,t_o)u\right)
    \\
    & &
    +
    d\left(F^{\epsilon_\nu}(t_2,t_o+t_1)\circ P(t_1,t_o)u,
      P(t_2,t_o+t_1)\circ P(t_1,t_o)u\right)
    \\
    & &
    \to 0 \ \ \mbox{ as }\ \ \nu \to +\infty\,.
  \end{eqnarray*}
  Hence taking the limit as $\nu\to+\infty$ in~\Ref{eq:nulimit} we
  obtain~\Ref{eq:process} for any $t_2$ with
  $\frac{t_2}{t_1}\in\mathbb{Q}$. The continuity of $P$ concludes the
  proof of~\Ref{eq:process}. Now, if $u \in \mathcal{D}_{t_o}$ then,
  by definition, $P(t_1,t_o)u \in
  \mathcal{D}^2_{t_o+t_1}$. But~\Ref{eq:process} implies also that
  \begin{displaymath}
    P(t,t_o+t_1)\circ P(t_1,t_o)u=P(t+t_1,t_o)u\in\mathcal{D}^2_{t_o+t+t_1}
  \end{displaymath}
  for any $t\geq 0$ such that $t_o+t+t_1\in I$, therefore
  $P(t_1,t_o)u\in \mathcal{D}_{t_o+t_1}$
  proving~\Ref{eq:process0}. Finally, the Lipschitz continuity with
  respect to $t$ and $u$ follows from Remark~\ref{remark4.6}, while
  the Lipschitz continuity with respect to $t_o$ is a direct
  consequence of the semigroup property. Indeed, take $0\leq t_1\leq
  t_2\leq T$, $u\in\mathcal{D}_{t_1} \cap \mathcal{D}_{t_2}$ and use
  the Lipschitz continuity with respect to $(t,u)$:
  \begin{eqnarray*}
    d\left(P\left(t,t_1\right)u,P\left(t,t_2\right)u\right)
    & = &
    d \left(
      P\left(t,t_2\right)\circ P\left(t_2-t_1,t_1\right)u,
      P\left(t,t_2\right)u
    \right)
    \\
    & \leq & 
    L \cdot d\left(P\left(t_2-t_1,t_1\right)u,u\right)
    \\
    & \leq & 
    L\cdot\Lip(F) \cdot (t_2-t_1) \,.
  \end{eqnarray*}
\end{proof}

The following Lemma concludes the proof of Theorem~\ref{thm:main}.

\begin{lemma}
  The map $P:\mathcal{A}\to X$ defined in Corollary~\ref{cor:Pi}
  restricted to the set $\mathcal{A}$ defined in Lemma~\ref{lemma4.7}
  satisfies~\ref{fgsfssald3}) in Theorem~\ref{thm:main}.
\end{lemma}

\begin{proof}
  When $t\in \left]0,\delta \right]$ one has
  $F^t(t,t_o)u=F(t,t_o)u$. Hence the Lemma is proved taking $t\in
  \left]0, \delta \right]$, $m=0$ and $n\to+\infty$ in
  Lemma~\ref{lem:integral}.
\end{proof}

\begin{proofof}{Proposition~\ref{prop:Euler}}
  Use the Lipschitz continuity of $P$ and~\Ref{eq:tangent}:
  \begin{eqnarray*}
    & &
    d \left( F^E (t,t_o) u, P(t,t_o) u \right)
    \leq
    \\
    & \leq &
    \sum_{k=1}^N
    d\left(
      \begin{array}{c}
        P(t-\tau_k,t_o+\tau_k) \, F^E (\tau_k,t_o)u,
        \\
        P(t-\tau_{k-1},t_o+\tau_{k-1}) \, F^E (\tau_{k-1},t_o)u
      \end{array}
    \right)
    \\
    & \leq &
    L \cdot
    \sum_{k=1}^N
    d\left(
      \begin{array}{c}
        F (\tau_k-\tau_{k-1},t_o+\tau_{k-1}) \,
        F^E (\tau_{k-1},t_o)u,
        \\
        P(\tau_k-\tau_{k-1},t_o+\tau_{k-1}) \,
        F^E (\tau_{k-1},t_o)u
      \end{array}
    \right)
    \\
    & \leq &
    \frac{2L^2}{\ln 2}\cdot
    \sum_{k=1}^N
    (\tau_k-\tau_{k-1}) 
    \int_0^\Delta \frac{\omega(\xi)}{\xi} \, d\xi
    \\
    & \leq &
    \frac{2L^2}{\ln 2} \, t \,
    \int_0^\Delta \frac{\omega(\xi)}{\xi} \, d\xi
  \end{eqnarray*}
\end{proofof}

\begin{proofof}{Proposition~\ref{prop:no}}
  Since we assume that~\textbf{(D)} is satisfied, we have no problem
  with the domains, i.e.
  \begin{equation}
    \label{eq:46}
    F(t,t_o) \mathcal{D}_{t_o} \subseteq \mathcal{D}_{t_o+t} \,,
    \qquad\mbox{ for } t \in [0,\delta] \,.
  \end{equation}
  By Remark~\ref{rem:basta2},~$F^{t2^{-m}}\left(t,t_o\right)u$
  converges uniformly to a map $P\left(t,t_o\right)u$ which is
  continuous in $(t,t_o)$ and uniformly Lipschitz in $u$. With $m=0$
  $n\to\infty$ in Lemma~\ref{lem:integral} we get immediately the
  tangency condition. Hence, because of~\Ref{eq:46} $P$ has all the
  properties of the process in Theorem~\ref{thm:main} except the
  Lipschitz continuity with respect to $(t,t_o)$ and the ``semigroup''
  property~\Ref{eq:process}. Finally, since
  \begin{displaymath}
    F^{2^{-n}t}\left(t,t_o+t\right)\circ F^{2^{-n}t}\left(t,t_o\right)u
    =
    F^{2^{-n}t}\left(2t,t_o\right)=
    F^{2^{-(n+1)}\cdot 2t}\left(2t,t_o\right)
  \end{displaymath}
  as $n\to +\infty$ we get~\Ref{it:no:2}.
\end{proofof}

\begin{lemma}
  \label{lem:convex}
  If $C$ is a closed and convex subset of the Hilbert space $X$ that
  satisfies~\textbf{(C)}, then for all $u \in C$, $v \in X$ and
  $\tau,\tau' \geq 0$,
  \begin{equation}
    \label{eq:StopCondition}
    \norma{
      \Pi \left(\Pi(u+\tau v) + \tau'v\right) - 
      \Pi\left(u+(\tau+\tau')v\right)
    }
    \leq 
    K \, \norma{v} \, \tau \, \tau' \,.
  \end{equation}
\end{lemma}

\begin{proof}
  Assume first that $\norma{v}=1$.

  If $u+\tau v \in C$, the right hand side above vanishes and the
  inequality trivially holds.

  If $u+\tau v \not\in C$, then there exists a unique $\tau'' \in
  \left] 0, \tau \right[$ such that $u + (\tau-\tau'') v \in \partial
  C$. Moreoever, setting $u'' = u + (\tau-\tau'') v$,
  \begin{eqnarray*}
    \Pi(u+\tau v) 
    & = & 
    \Pi\left( u'' + \tau'' v \right)
    \\
    \Pi\left( u + (\tau+\tau') v \right)
    & = &
    \Pi\left( u'' + ( \tau'' + \tau') v \right)
  \end{eqnarray*}
  By the convexity of $C$, both $u'' + \tau'' v$ and $u'' + ( \tau'' +
  \tau') v$ are in $X \setminus C$. Hence, by~\textbf{(C)}, the
  map
  \begin{displaymath}
    (\tau',\tau'') 
    \mapsto 
    \norma{
      \Pi \left( \Pi\left( u'' + \tau'' v \right) + \tau' v \right)
      -
      \Pi\left( u'' + ( \tau'' + \tau') v \right)
    }
  \end{displaymath}
  is $\C2$ for $\tau',\tau'' \in [0,\delta]$. It vanishes both for
  $\tau'=0$ and for $\tau''=0$, so there exists a positive $K$ such
  that
  \begin{eqnarray*}
    & &
    \norma{
      \Pi \left(\Pi(u+\tau v) + \tau'v\right) - 
      \Pi\left(u+(\tau+\tau')v\right)
    }
    =
    \\
    & = &
    \norma{
      \Pi \left( \Pi\left( u'' + \tau'' v \right) + \tau' v \right)
      -
      \Pi\left( u'' + ( \tau'' + \tau') v \right)
    }
    \\
    & \leq &
    K \, \tau' \, \tau''
    \\
    & \leq &
    K \, \tau \, \tau' \,.
  \end{eqnarray*}
  The case $\norma{v} \neq 1$ follows by a straightforward
  rescaling procedure.
\end{proof}

\noindent\textbf{Acknowledgment.} The authors thank Fabio Zucca for
useful hints in the construction of the function $f$ in
Paragraph~\ref{sub:example}.

\def\cprime{$'$}

\end{document}